\documentclass[twoside,11pt]{article}
\usepackage{jmm}

\begin{document}
\Title{A globally convergent gradient-like method based on the Armijo line search}

\Author{Ahmad Kamandi$^\dag$\correspond,  Keyvan Amini$^{\ddag}$}

\Address{$^{\dag}$ Department of Mathematics ,University of Science and Technology of Mazandaran, Behshahr, Iran \\ 
         $^{\ddag}$  Department of Mathematics, Razi university, Kermanshah, Iran\\}

\Email{ahmadkamandi@mazust.ac.ir, kamini@razi.ac.ir}
\Markboth{F. Author, S. Author}{An iterative method to solve ...}

\Abstract {In this paper,  a new conjugate gradient-like algorithm is proposed to solve unconstrained optimization problems. The step directions generated by the new algorithm satisfy sufficient descent condition independent of the line search. The global convergence of the new algorithm, with the Armijo backtracking line search,  is proved.  Numerical experiments indicate the efficiency and robustness of the new algorithm.
}

\Keywords{Unconstrained optimization,  Conjugate gradient algorithm, Global convergence,  Armijo condition}
\AMS{2020}{65K05 , 65K10.}

\section{Introduction}
Motivated by numerous real-world applications such as machine learning \cite{b1,b2}, big data \cite{c2}, and compressive sensing \cite{c1,s1}, it is essential to tackle large scale optimization problems and to design adapted efficient and robust algorithms, which are computationally tractable and globally convergent.

Consider the following unconstrained optimization problem,
\begin{equation}\label{e1} 
\begin{array}{ll}
\min          &~ f(x)\\
\mathrm{s.t.} &~ x\in \mathbb{R}^n,
\end{array}
\end{equation}
where $f:\mathbb{R}^n\rightarrow \mathbb{R}$ is a  continuously differentiable function.  Among numerical algorithms for solving \eqref{e1}, line search algorithms start from an initial guess, $x_0$, for the solution of the problem and generate a sequence of iterations $\{x_k\}$ by the following recurrence relation,
 \begin{equation}\label{e2} 
x_{k+1}=x_k+\alpha_k d_k,
\end{equation}
where $\alpha_k$ is the steplength and  $d_k$ is the search direction at iteration $k$. These algorithms try  to find a new iterate $x_{k+1}$ with a lower function value than $x_k$. To this aim the direction $d_k$ needed to be a downhill direction. The steepest descent  method is a well-known line search method that uses the direction $d_k=-\nabla f(x_k)$, along which $f$ decreases most rapidly, at every iteration.

Having the direction $d_k$,  the ideal choice for the steplength $\alpha_k$ would be the global minimizer of
\begin{equation}\label{e5}
f(x_k+\alpha d_k),~~~~ \alpha>0.
\end{equation}
The procedure of finding exact minimizer of \eqref{e5} is called exact line search and, in general, it is computationally expensive. 
More practical strategies perform an inexact
line search to identify a steplength that achieves adequate reductions in $f$ at minimal cost \cite{n1}. A popular inexact line search condition is,
\begin{equation}\label{e6}
f(x_k+\alpha_k d_k) \leq f(x_k)+c_1 \alpha_k d_k^Tg_k,
\end{equation}
where $g_k=\nabla f(x_k)$  and  $c_1\in (0,1)$. Inequality \eqref{e6} is sometimes called the Armijo condition. An appropriate steplength $\alpha_k$ satisfying Armijo condition can  be found  by a  so-called backtracking  approach.\\

Another important procedure to find the steplength is known as the Wolfe line search  for which the steplength $\alpha_k$ must satisfy the following conditions,
\begin{equation*}
\begin{array}{l}
f (x_k+\alpha_k d_k) \leq f(x_k)+c_1 \alpha_k d_k^Tg_k, \\
\nabla f (x_k+\alpha_k d_k) ^T d_k \geq  c_2 d_k^T g_k,
\end{array}
\end{equation*}
where $0<c_1<c_2<1$.  Strong Wolfe conditions are also a modification of the Wolfe conditions that  require $\alpha_k$ to satisfy,
\begin{equation*}
\begin{array}{l}
f (x_k+\alpha_k d_k) \leq f(x_k)+c_1 \alpha_k d_k^Tg_k,  \\
| \nabla f (x_k+\alpha_k d_k) ^T d_k |\leq  c_2 |d_k^T g_k |.
\end{array}
\end{equation*} 

It is generally accepted that the steepest descent method is badly affected by ill-conditioning. The behaviour of this method is investigated  comprehensively  for a two dimensional case in \cite{s2}. Conjugate gradient (CG) methods are an important class  of line search methods for solving \eqref{e1}; especially for the case that the dimension of the problem is large. The search direction in CG methods is determined by
\begin{equation}\label{e3}
d_k=\left\{
\begin{array}{ll}
-g_0 &~~ \textrm{if}~ k=0, \\
-g_k+ \beta_k  d_{k-1} &~~ \textrm{if}~ k \geq 1,
\end{array}  
\right.
\end{equation}
where  $\beta_k$  is a parameter that is called CG parameter and $g_k$ is the gradient of the objective function at $x_k$. 

The first nonlinear conjugate gradient method was introduced by Fletcher and Reeves in the 1960s \cite{f1}. The main difference among CG methods is in the formulas of computing their parameter for example,
\begin{equation*}
\begin{array}{ll}
\beta_k^{HS}=\frac{g_k^Ty_{k-1}}{d_{k-1}^Ty_{k-1}}~~\text{ \cite{h2}}, & \beta_k^{FR}=\frac{\|g_k\|^2}{\|g_{k-1}\|^2}~~  \text{\cite{f1}},\\
\beta_k^{PRP}=\frac{g_k^Ty_{k-1}}{\|g_{k-1}\| ^2} ~~\text{ \cite{p1}},&\beta_k^{CD}=\frac{\|g_{k}\|^2}{-d_{k-1}^Tg_{k-1}} ~~ \text{ \cite{f0}},\\
\beta_k^{LS}=\frac{g_k^Ty_{k-1}}{-d_{k-1}^Tg_{k-1}} ~~\text{ \cite{l1}},&\beta_k^{DY}=\frac{\|g_k\|^2}{d_{k-1}^Ty_{k-1}} ~~\text{ \cite{d1}},\\
\beta_k^{N}=(y_{k-1}-2d_{k-1}\frac{\|y_{k-1}\|^2}{d_{k-1}^Ty_{k-1}})^T\frac{g_k}{d_{k-1}^Ty_{k-1}} ~~\text{ \cite{h0}}.&\\
\end{array}  
\end{equation*}
where $y_{k-1}=g_k-g_{k-1}$. 

When the objective function $f$ is a strongly convex quadratic all of the above CG parameters are equivalent with an exact line search and the CG methods, obtained from them, converge to the global minimizer of $f$ in at most $n$ iterates. The global convergence properties of these  and some other CG methods  are  reviewed in  \cite{h1}.

The Wolfe and the strong Wolfe conditions play an important role in establishing the global convergence of many CG methods \cite{d1,g00,l01,l2,y1}.
But finding the steplength that satisfies these conditions, need some additional gradient evaluations. So, the Wolfe and the strong Wolfe line search are more expensive than the Armijo line search.  

Another important property of the step direction, in convergence analysis of a CG method, is  the "sufficient descent" condition. A direction $d_k$ is called a sufficient descent direction if there exists a positive parameter $c$ such that,
\begin{equation}\label{sd}
d_k^Tg_k \leq -c \|g_k\|^2.
\end{equation}

 The pioneer work about the global convergence of FR method with inexact line search was proposed by Al-Baali \cite{a1}. He proved that FR method generates sufficient descent directions and this method is globally convergent when  the steplength satisfy the strong Wolfe conditions  with $0<c_1<c_2<\frac{1}{2}$.  Liu et al. \cite{g1} and Dai and Yuan \cite{d1}
extended this result to $c_2=\frac{1}{2}$. It is shown that FR
method with the strong Wolfe line search may not be a descent direction for the case that $c_2>\frac{1}{2}$ \cite{d2}.

For FR method, neither  Armijo line search nor  Wolfe line search,  guarantee that the directions generated by this method are sufficient descent directions. In 2006 \cite{z1}, Zhang  et al.  proposed a modified FR conjugate gradient method.  In their method, the direction $d_k$ is computed by,
\begin{equation}\label{e7}
d_k=\left\{
\begin{array}{ll}
-g_0 &~~ \textrm{if}~ k=0, \\
-\theta_k g_k+ \beta_k^{FR}  d_{k-1} &~~ \textrm{if}~ k \geq 1,
\end{array}  
\right.
\end{equation}
where,
\begin{equation*}
\theta_k = \frac{d_{k-1}^Ty_{k-1}}{\|g_{k-1}\|^2},
\end{equation*}

From the definition of $d_k$ by \eqref{e7} one can easily get $g_k^Td_k=-\|g_k\|^2$.  Zhang  et al.  , proved that the
their method is globally convergent when the steplength, $\alpha_k$,  satisfies the following Armijo-type condition,
\begin{equation}\label{a1}
f(x_k+\alpha d_k) \leq f(x_k)+\sigma_1 \alpha d_k^Tg_k-\sigma_2 \alpha^2 \|d_k\|^2,
\end{equation}
where $\sigma_1$ and $\sigma_2$ are some positive constants.

In this paper,inspired by the work of Zhang  et al. \cite{z1},  we propose a new conjugate gradient-like method. The step directions generated by the new method are sufficient descent directions independent of the line search is used to compute the steplength.  The new CG-like method is globally convergent under the Armijo condition. Numerical tests indicate the efficiency of the new method in solving a collection of unconstrained test problems from CUTEst package.

The rest of this paper is organized as follows. In the next section, the new CG-like method is introduced. Global convergence of the new method is analyzed  in Section 3. Numerical results are reported in Section 4. Some conclusions are made in Section 5.
\section{The new method}
In this section, based on the pervious section discussion, we propose a new CG-like method. The sequence of iteration $x_k$ in the new method is obtained from \eqref{e2} for which the direction $d_k$ is computed by \eqref{e3}. While the parameter $\beta_k$ in the new method is computed by,
\begin{equation}\label{e9}
\beta_k^{new}=\tau \frac{\|g_k\|}{\|d_{k-1}\|},
\end{equation}
where $\tau \in (0,1)$.  Note that, by $\tau= 0$ the new method is equivalent to the steepest descent method.

 By the new parameter \eqref{e9} we follow two goals: Firstly, to obtain a computationally inexpensive conjugate gradient like direction with sufficient descent property \eqref{sd}, in order to establish the well known Zoutendijk condition \cite{z2} when the steplength satisfies the Armijo condition \eqref{e6}. Secondly, to make the norm of the direction $d_k$, obtained from the new parameter, bounded by a fixed coefficient of $\|g_k\|$ in order to derive the global convergence of the new CG-like method.

Note that, for the direction $d_k$ defined by \eqref{e3}, with the CG parameter computed by \eqref{e9}, we have,
\begin{equation*}
d_k^Tg_k=-\|g_k\|^2+\tau \frac{\|g_k\|}{\|d_{k-1}\|}d_{k-1}^Tg_k,
\end{equation*}
by the Cauchy-Schwarz inequality, it can be concluded that,
\begin{equation}\label{e10}
d_k^Tg_k \leq -\|g_k\|^2+\tau \|g_k\|^2=-(1-\tau) \|g_k\|^2.
\end{equation}
So, the new direction $d_k$ is a sufficient descent direction independent of the line search. For this direction we also have,
\begin{equation}\label{e11}
\|d_k\| \leq (1+\tau)\|g_k\|.
\end{equation}
Relations \eqref{e10} and \eqref{e11} illustrates two basic properties of the new CG-like direction which are provided by the new parameter \eqref{e9}. These properties will be used to prove the global convergence of the new method. 

 In the new CG-like method, the steplength $\alpha_k$ is determined such that satisfies the Armijo condition. To this aim, we  use a  backtracking  approach to compute the steplength.
Now we are ready to propose the algorithm of the new CG-like method\\
\rule{7.5cm}{.4pt}\\
\textbf{Algorithm 1} (New CG-like method)\\
\rule{7.5cm}{.4pt}\\
\textbf{Step 0} Given positive constants $c_1$, $\epsilon$, $\tau \in (0,1)$. Choose an initial point $x_0 \in  \mathbb{R}^n$,  compute $g_0$ and set $k=0$.\\
\textbf{Step 1} If $\|g_k\|\leq \epsilon$ stop.\\
\textbf{Step 2} If $k=0$ set $d_k=-g_k$. Otherwise, compute the parameter  $\beta_k^{new}$ from \eqref{e9} and the direction $d_k$ from \eqref{e3}.  \\
\textbf{Step 3} Compute the steplength $\alpha_k$ that satisfies \eqref{e6}.\\
\textbf{Step 4} Set $x_{k+1}=x_k+\alpha_k d_k$. Increase $k$ by one, compute $g_k$ and go to Step 1.\\
\rule{7.5cm}{.4pt}\\
\section{Convergence properties}
In this section, we analyze  the global convergence of the new CG-like method. To this aim, similar to \cite{z1},  we made the following assumptions: 
\begin{description}
\item[\textbf{(H1)}] The objective function $f(x)$ has a lower bound on the level set,
\[\mathcal{L}=\{x\in \mathbb{R}^n ~|~ f(x)\leq f(x_0),\ x_0\in \mathbb{R}^n\}.\]
\item[\textbf{(H2)}]  The objective function $f(x)$ is continuously
differentiable and its gradient is Lipschitz continuous on a neighbourhood $N$ of $\mathcal{L}$, namely, there exists a constant $L > 0$ such
that
\[\| g(x)-g(y)\|\leq L \| x-y\|,~~~~ \forall x,y \in N.\]
\end{description}
The following lemma provides a lower bound for the steplength $\alpha_k$ (generated by  Algorithm 1). The result of this lemma will be needed in the rest of this section.
\begin{lemma}\label{lm1}
Let the steplength $\alpha_k$ be generated by Algorithm 1. Then, under the assumptions H1 and H2, there exists a positive constant $C$ such that,
\begin{equation}\label{e12}
\alpha_k \geq C \frac{\|g_k\|^2}{\|d_k\|^2}.
\end{equation}
\end{lemma}
\begin{proof}
There are two possible cases for the steplength $\alpha_k$ generated by  Algorithm 1,
\begin{itemize}
\item[Case 1:] $\alpha_k=\bar{\alpha}$. In this case, from \eqref{e10}, we have,
\begin{equation*}
\alpha_k \|d_k\| \|g_k\| \geq -\alpha_k d_k^Tg_k =-\bar{\alpha} d_k^Tg_k  \geq \bar{\alpha}(1-\tau) \|g_k\|^2 ,
\end{equation*} 
So, 
\begin{equation*}
\alpha_k\geq \bar{\alpha}(1-\tau) \frac{\|g_k\|}{ \|d_k\|},
\end{equation*} 
the squre of this inequality along with  $\alpha_k=\bar{\alpha}$ yeild that \eqref{e12} satisfies with $C=\bar{\alpha}(1-\tau)^2$.
\item[Case 2:] $\alpha_k<\bar{\alpha}$. In this case, by the fact that $\rho^{-1} \alpha_k$ does not satisfy \eqref{e6}, we have,
\begin{equation}\label{e13}
f (x_k+\rho^{-1}\alpha_k d_k)> f(x_k)+c_1 \rho^{-1} \alpha_k d_k^Tg_k,
\end{equation}
by the mean value theorem there exists a $t_k \in (0,1)$ such that
\begin{equation}
\begin{array}{ll}
f (x_k+\rho^{-1}\alpha_k d_k)- f(x_k)&\\
=\rho^{-1} \alpha_k \nabla f (x_k+t_k \rho^{-1} \alpha_k d_k)^T d_k&\\
=\rho^{-1} \alpha_k d_k^T g_k+\rho^{-1} \alpha_k[\nabla f (x_k+t_k \rho^{-1} \alpha_k d_k)-g_k]^T d_k.
\end{array}
\end{equation}
This equation along with the assumption H2 result that,
 \begin{equation}\label{e14}
f (x_k+\rho^{-1}\alpha_k d_k)- f(x_k) \leq \rho^{-1} \alpha_k d_k^T g_k+L \rho^{-2} \alpha_k^2 \|d_k\|^2 .
\end{equation}
From \eqref{e10}, \eqref{e13} and \eqref{e14}  we can conclude that,
\begin{equation*}
\alpha_k > \frac{\rho (1-c_1)(1-\tau)\|g_k\|^2}{L \|d_k\|^2}
\end{equation*}
This inequality means that \eqref{e12} satisfies with $C=\frac{\rho (1-c_1)(1-\tau)}{L}$.
\end{itemize}
By the two above cases, the inequality \eqref{e12} is always valid with $$C=min\{\bar{\alpha}(1-\tau)^2,\frac{\rho (1-c_1)(1-\tau)}{L}\}$$
So, the proof is compeleted.
\end{proof}
The next lemma is known as Zoutendijk condition \cite{z2}.
\begin{lemma}
 Suppose that H1 and H2 hold and $d_k$ is generated by  Algorithm 1, then 
\begin{equation}\label{e15}
\sum_{k=0}^{\infty} \frac{\|g_k\|^4}{\|d_k\|^2} < \infty.
\end{equation}
\end{lemma}
\begin{proof}
From \eqref{e6}  for any $k \geq 0$ we have,
\begin{equation}\label{e16}
\sum_{j=0}^{k} -\alpha_j d_j^Tg_j \leq f(x_0)-f(x_{k+1}),
\end{equation}
By the fact that the direction $d_k$ is a sufficient descent direction and the Armijo condition is valid for each iteration, the sequence $\{f(x_k)\}$ is decreasing. So, the assumption H1 results that this sequence is convergent. Therefore, by taking limit from the inequality  \eqref{e16}, when ~$k \rightarrow \infty$, we have,  
\begin{equation}\label{ee}
\sum_{j=0}^{\infty} -\alpha_j d_j^Tg_j< \infty.
\end{equation}

From the inequalities \eqref{ee}, \eqref{e10} and \eqref{e12} we have,
\begin{equation}
\sum_{j=0}^{\infty} C (1-\tau) \frac{\|g_j\|^4}{\|d_j\|^2}< \infty.
\end{equation}
So, the proof is completed.
\end{proof}
Now, by the inequality \eqref{e15} and  \eqref{e11} we can easily conclude that,
\begin{equation}\label{eg}
\sum_{k=0}^{\infty} \|g_k\|^2< \infty.
\end{equation}
The inequality \eqref{eg}  results the global convergence of the new CG-like method. 
\begin{theorem}
 Suppose that H1 and H2 hold and the sequence $\{x_k\}$ is generated by  Algorithm 1, then
\begin{equation*}
\lim_{k \rightarrow \infty} \|g_k\|=0.
\end{equation*} 
\end{theorem}
\section{Numerical Results}
In this section, we report some numerical experiments that indicate the efficiency of Algorithm 1. To this aim we implement  Algorithm 1 (NEW) ,  The conjugate gradient algorithm with the CG parameter proposed by Hager and Zhang \cite{h0} (HZ),  Fletcher and Reeves (FR) algorithm and the modified Fletcher and Reeves (MFR) algorithm \cite{z1} in MATLAB environment on a laptop (CPU Corei7-2.5 GHz, RAM 12 GB) and compare their results  obtained from solving a collection of 243 unconstrained optimization test problems from CUTEst collection \cite{g0}. The test problems  and their dimension are listed in Table 1. 

For all the considered algorithms we set $\alpha_k=max\{\bar{\alpha}\rho^i : i=0,1,2,...\}$ which satisfies Armijo line search condition.  For the initial steplength $\bar{\alpha}$, similar to \cite{z1}, we set   $\bar{\alpha}=\frac{s_{k-1}^Ts_{k-1}}{s_{k-1}^Ty_{k-1}}$ if $s_{k-1}^Ty_{k-1}>10^{-8}$ and $\bar{\alpha}=1$ otherwise, where $s_{k-1}=g_k-g_{k-1}$. The other parameters are chosen as $\rho=0.5$, $c_1=10^{-4}$ and  $\tau= 0.002$.

Note that, in order to find an appropriate value for the parameter $\tau$, we tested the numerical behavior of the new algorithm for different value of this parameter.  Among  \{0, 0.001, 0.002, 0.003, 0.004, 0.005, 0.01, 0.02, 0.03, 0.04, 0.05, 0.1, 0.2, 0.3, 0.4, 0.5, 0.6, 0.7, 0.8, 0.9\}, the best result is obtained for $\tau= 0.002$.

\renewcommand{\arraystretch}{1.1}
{\small
\begin{longtable}{llll}
\multicolumn{2}{l}%
{Table 1. List of test problems}\\[5pt]
\hline 
\multicolumn{1}{l}{Problem name} &\multicolumn{1}{c}{Dim}&
\multicolumn{1}{l}{Problem name} &\multicolumn{1}{c}{Dim}\\
\hline
\endfirsthead
\multicolumn{2}{l}%
{Table 1. (\textit{continued})}\\[5pt]
\hline
\endhead
\hline
\endfoot
\endlastfoot
ARGLINB  & 50,100,200 & ARGLINC   &50,100,200 \\
BDQRTIC & 100,500,1000,5000 & BROWNAL & 100,200,1000\\
BRYBND  & 50,100,500 & CHNROSNB &  50\\
CHNRSNBM & 50& EIGENALS &110\\
EIGENBLS & 110&ERRINROS& 50\\
ERRINRSM &50 &EXTROSNB &100,1000\\
FREUROTH & 50,100,500,1000,5000 & LIARWHD  & 100,500,1000,5000\\
MANCINO  & 50,100 & MODBEALE & 200,2000\\
MSQRTALS & 100 & MSQRTBLS &100\\
NONDIA   & 50,90,100,500,1000,5000&NONSCOMP &50,100,500,1000,5000\\
OSCIGRAD & 100,1000 & OSCIPATH &100,500\\
PENALTY1  &50,100,500,1000 & PENALTY2 & 50,100,200 \\
SPMSRTLS & 100,499,1000,4999 & SROSENBR  &50,100,500,1000,5000\\
TQUARTIC  &50,100,500,1000,5000 &VAREIGVL & 50,100,500,1000,5000\\
WOODS  &  100,1000,4000 & ARWHEAD &100,500,1000,5000\\
BOX   &   100 & BOXPOWER &100,1000 \\
BROYDN7D &  50,100,500,1000 & COSINE &   100,1000 \\
CRAGGLVY & 50,100,500,1000,5000 & DIXMAANA  & 90,300,1500,3000 \\
DIXMAANC &  90,300,1500,3000 & DIXMAAND  & 90,300,1500,3000 \\
DIXMAANE & 90,300,1500,3000 &DIXMAANF & 90,300,1500,3000 \\
DIXMAANG &  90,300,1500,3000 &DIXMAANH  &90,300,1500,3000\\
DIXMAANI  &90,300,1500,3000 &DIXMAANJ & 90,300,1500,3000 \\
DIXMAANK  & 90,300,1500,3000 & DIXMAANL & 90,300,1500,3000\\
DIXMAANM & 90,300,1500,3000 & DIXMAANN & 90,300,1500,3000 \\
DIXMAANO & 90,300,1500,3000 & DIXMAANP  &90,300,1500,3000 \\
DQRTIC   & 50,100,500,1000, 5000 &EDENSCH& 2000\\
ENGVAL1  & 50, 100,1000,5000 & FLETCHCR  & 1000 \\
FMINSURF &  64,121,961,1024 & INDEFM &   50 \\
NCB20B  & 50,1000, 2000 & NONCVXU2 & 100,1000,5000 \\
NONCVXUN & 100,1000,5000 & NONDQUAR & 100,1000,5000\\
PENALTY3 & 50,100 & POWELLSG  & 60,80,100,500,1000,5000 \\
POWER  &  50,75,100,500,1000,5000 & QUARTC &  100,500,1000,5000 \\
SCHMVETT  & 100,500,1000,5000 & SINQUAD &  50,100\\
SPARSINE  & 50,100 & SPARSQUR & 50,100,1000,5000 \\
TOINTGSS  & 50,100,500,1000,5000 & VARDIM    & 50,100,200\\
DIXON3DQ & 100 & DQDRTIC  & 50,100,500,1000,5000 \\
HILBERTB &  50 & TESTQUAD &1000,5000\\
TOINTQOR  &50 &TRIDIA  &  50,100,500,1000,5000\\

\hline
\end{longtable}
}

  In our experiments the stopping tolerance for all of the algorithms is
\begin{equation*}
\epsilon= 10^{-6}\|g_0\|.
\end{equation*}
Also, a failure is reported when,  the total number of iterations exceeds 4000 (Case i ) or  the steplength $\alpha_k$ become less than $eps/10$ (Case ii ). Where $eps$ is the floating-point relative accuracy in MATLAB software. In the results of 243 problems 74,6,50 and 7 failures are reported for  the FR, the New, the MFR and the HZ algorithms respectively.  Where many of the failures in the FR are related to  Case ii,  while for the others the failures are related to  Case i.

To visualize the whole behavior of the algorithms, we use the performance profiles proposed by
Dolan and More \cite{d3}. The total number of function evaluations, the total number of iterations  and the running time of each algorithm are considered  as performance indexes. Note that at each iteration of the considered algorithms the gradient of the objective function is computed just one time, so the total number of iterations and the total number of the gradient evaluations are the same.

 Fig.1  illustrates the performance profile of the algorithms, where the performance index is the total number of function evaluations. It can be seen that the NEW is the best algorithm with probability around 60\%, while the probability of solving a problem as the best algorithm are around 27\%,  18\%  and 8\% for the HZ, the MFR and the FR respectively.  
 
 The performance index in Fig.2  is  the total number of iterations. From this figure we observe that the NEW obtains the most wins on approximately 55\% of all test problems an the probability of being best algorithm is 30\%,  25\% and 8\% for  the HZ, the MFR and the FR respectively.
 
  \begin{figure} [ht]
 \begin{center}
   \includegraphics[width=6.5cm]{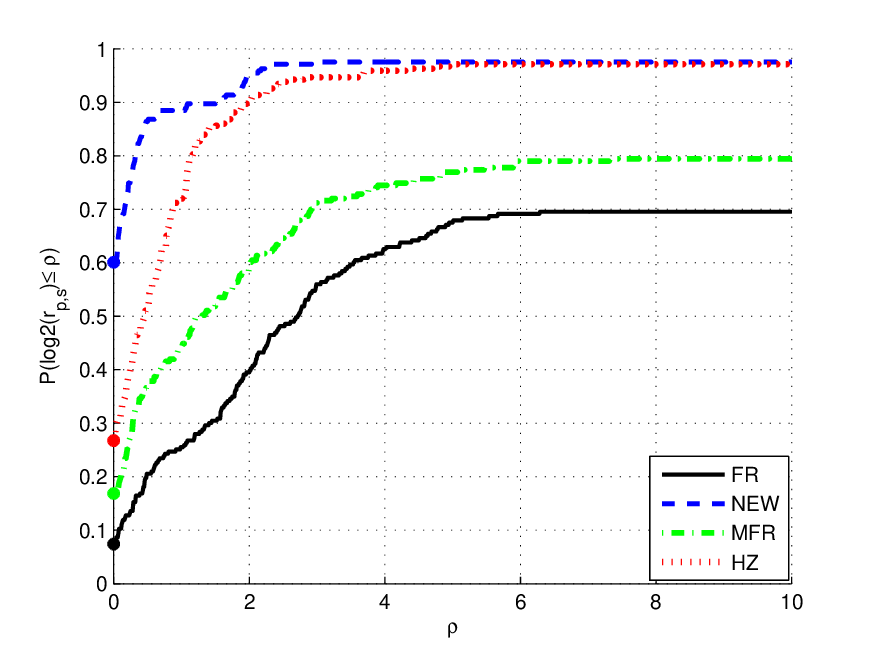}\\
  \vspace{-.3cm}
 \caption{Performance profiles for the number of function evaluations}
 \end{center}
\end{figure}
 
 \begin{figure} [ht]
 \begin{center}
   \includegraphics[width=6.5cm]{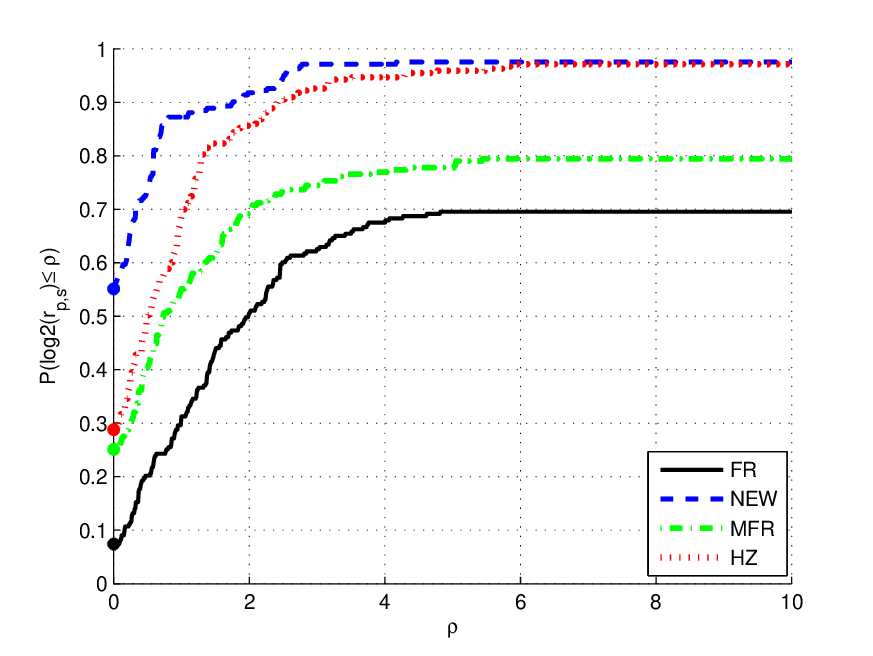}\\
  \vspace{-.3cm}
 \caption{Performance profiles for the number of iterations}
 \end{center}
\end{figure}

\begin{figure} [ht]
 \begin{center}
   \includegraphics[width=6.5cm]{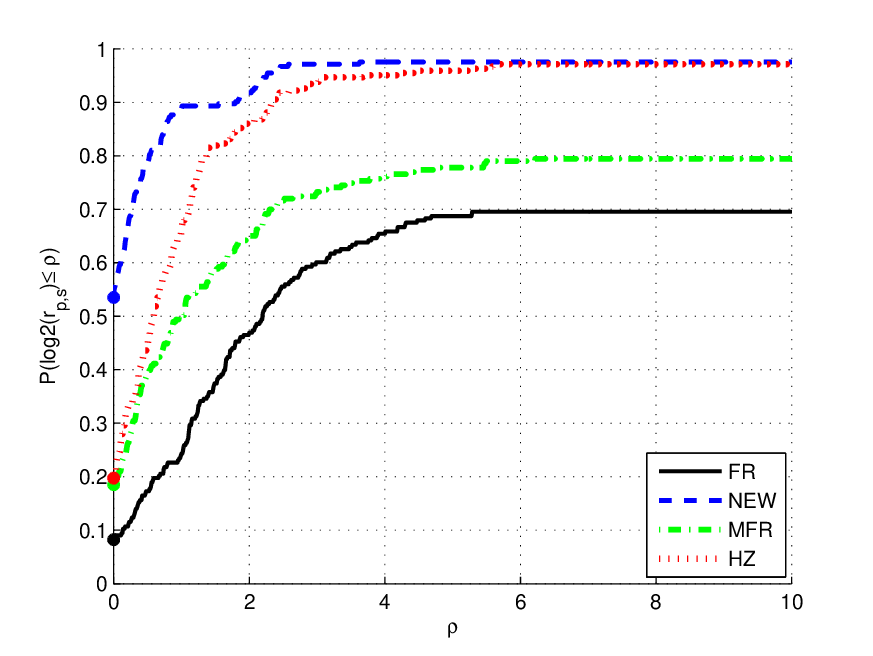}\\
  \vspace{-.3cm}
\caption{Performance profiles for the running times}
 \end{center}
\end{figure}
 
  The performance profiles for the running times are illustrated in Fig.3. From this figure, it can be observed that the NEW is the best algorithm. Another important factor of these three figures is that the graph of the NEW algorithm grows up faster than the other algorithms. 
  
 We also compare the mentioned algorithms in solving a set of convex quadratic problems from CUTEst package which are classified as QUR2. With the exact linesearch, as we expected, the FR and HZ methods perform the same and their results were much better than the NEW and the MFR. But, with the Armijo linesearch, the results were different.

 The graphs of Fig.4 ,Fig.5 and Fig.6  illustrate the performance of the four methods in solving 15 convex quadratic test problems which are listed in Table 2.  These graphs indicate that the new methods is efficient than the other method in solving convex quadratic test problems, when the step length is computed by the Armijo line search.
  
  From the presented results, we can conclude that the NEW algorithm is more efficient than the HZ, the MFR and the FR algorithms in solving unconstrained optimization problems.

 \begin{figure} [ht]
 \begin{center}
   \includegraphics[width=6.5cm]{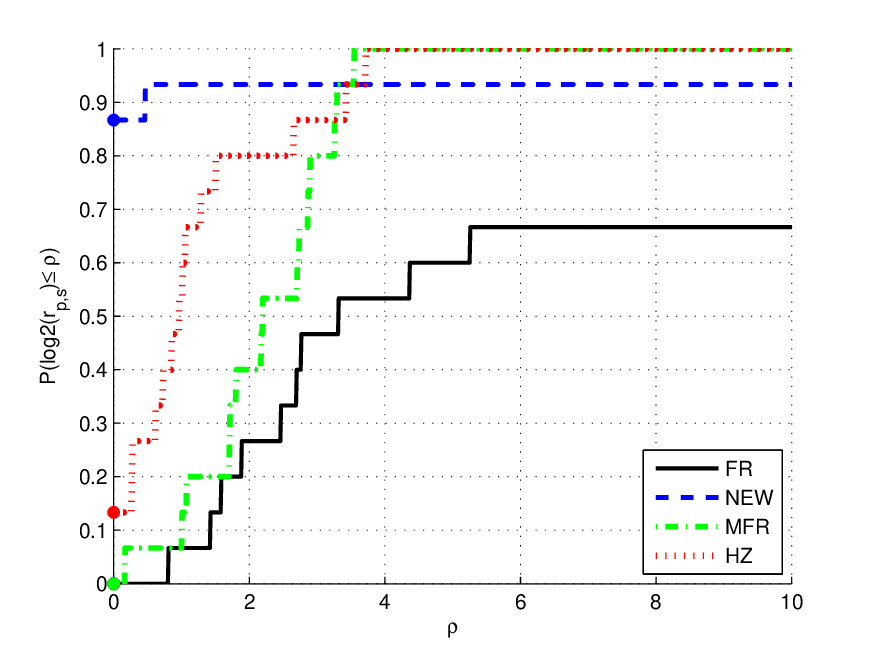}\\
 \vspace{-.3cm}
\caption{Performance profiles for the number of function evaluations (convex quadratic test problems)}
 \end{center}
\end{figure}

\begin{figure} [ht]
 \begin{center}
   \includegraphics[width=6.5cm]{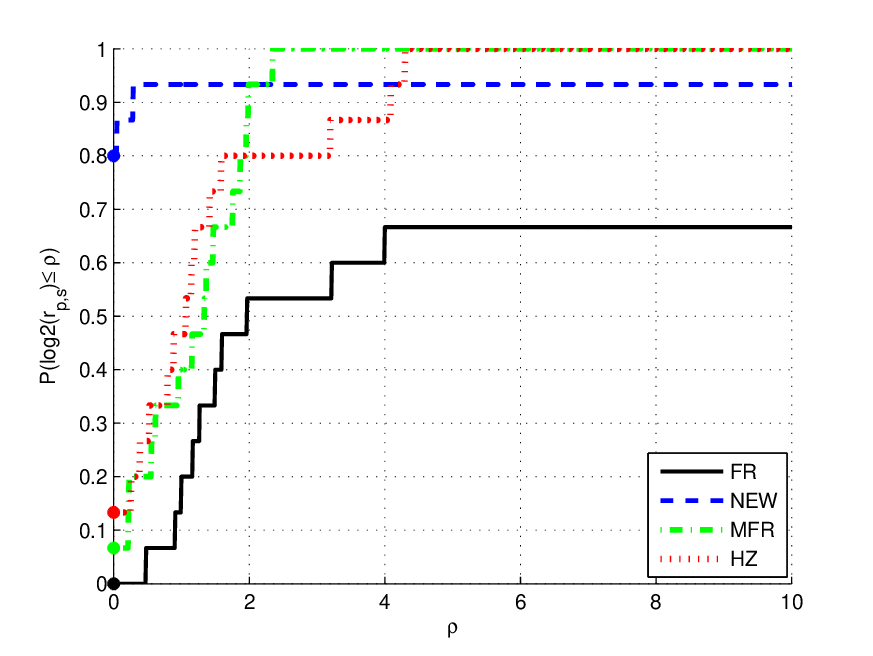}\\
  \vspace{-.3cm}
\caption{Performance profiles for  the number of iterations (convex quadratic test problems)}
 \end{center}
\end{figure}

\begin{figure} [ht]
 \begin{center}
   \includegraphics[width=6.5cm]{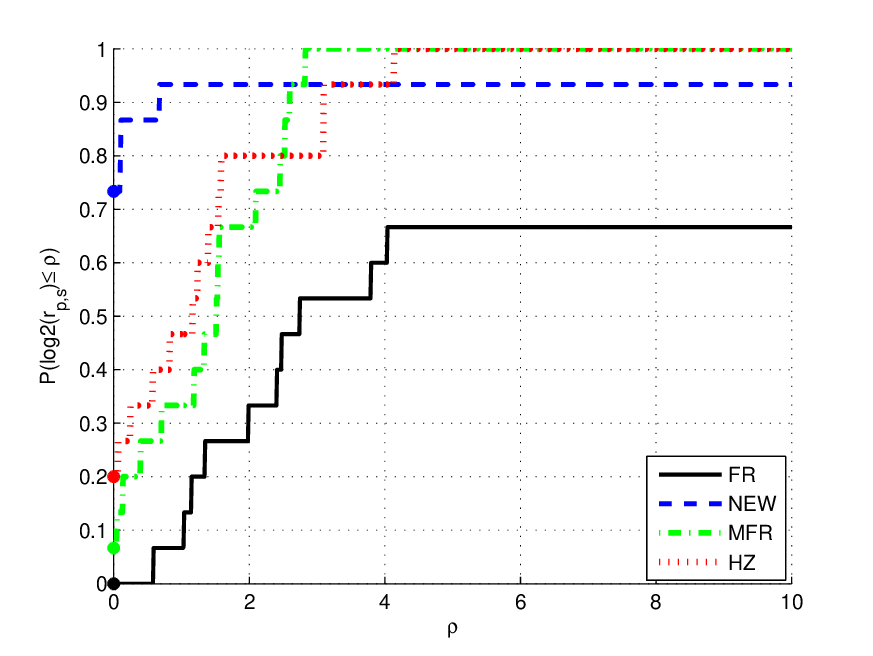}\\
  \vspace{-0.3cm}
\caption{Performance profiles for the running times (convex quadratic test problems)}
 \end{center}
\end{figure}

 \renewcommand{\arraystretch}{1.1}
{\small
  \begin{longtable}{llll}
\multicolumn{2}{l}%
{Table 2. List of convex quadratic test problems}\\[5pt]
\hline 
\multicolumn{1}{l}{Problem name} &\multicolumn{1}{l}{Dim}&
\multicolumn{1}{l}{Problem name} &\multicolumn{1}{l}{Dim}\\
\hline
\endfirsthead
DIXON3DQ  & 100 & TESTQUAD   &1000,5000 \\
DQDRTIC & 50,100,500,1000,5000 & TOINTQOR & 50\\
HILBERTB  & 50 & TRIDIA &  50,100,500,1000,5000\\
\hline
\end{longtable}
}

\section{Conclusion}
In this paper, we propose a new conjugate gradient-like algorithm.  The step directions generated by the new algorithm satisfies sufficient descent condition independent of line search. The global convergence of the new algorithm with the Armijo backtracking line search is investigated under some mild assumptions.  Numerical experiments showed the efficiency and robustness of the new algorithm for solving a collection of unconstrained optimization problems from CUTEst package. 

\newpage

\section*{Acknowledgements}


\begin{thebibliography}{99}
\bibitem{a1} Al-Baali  M. \emph{Descent property and global convergence of the
Fletcher-Reeves method with inexact line search}, IMA J. Numer. Anal. \textbf{5} (1985) 121-124.

\bibitem{b1}  Bottou  L. \emph{Large-scale machine learning with stochastic gradient descent}, Proceedings of COMPSTAT’ 2010. Springer, (2010), 177-186.

\bibitem{b2}  Bottou L., Curtis F., and Nocedal J. \emph{Optimization methods for large-scale machine learning}, SIAM Rev.  \textbf{60}(2) (2018),  223-311.

\bibitem{c1} Candes E. \emph{Compressive sampling}, Proceedings oh the
International Congress of Mathematicians: Madrid, August 22-30,  2006: invited lectures. (2006) 1433-1452.

\bibitem{c2} Cevher V., Becker S. and Schmidt M. \emph{Convex optimization for big data: Scalable, randomized, and parallel algorithms for big data analytics},  IEEE  Signal Process. Magzi. \textbf{31}(5) (2014),   32-43.

\bibitem{d1} Dai Y. and Yuan Y. \emph{A nonlinear conjugate gradient method
with a strong global convergence property}, SIAM J. Optim.  \textbf{10}(1) (1999),  177-182.


\bibitem{d2} Dai  Y. and Yuan Y. \emph{Nonlinear conjugate gradient methods}, Shanghai
Scientific. 2000.

\bibitem {d3}  Dolan E. and More J. \emph{ Benchmarking optimization software with performance profiles}, Math. Program.  \textbf{91}(2) (2002),  201-213.

\bibitem{f0} Fletcher R. \emph{Practical methods of optimization vol. 1:  Unconstrained optimization}, John Wiley and Sons, New York, (1987).

\bibitem{f1} Fletcher R. and Colin R.  \emph{Function minimization by conjugate gradients}, The computer journal  \textbf{7}(2) (1964),  149-154.


\bibitem{g00} Gilbert J.C and  Nocedal J. \emph{Global convergence properties
of conjugate gradient methods for optimization}, SIAM J. Optim.   \textbf{2}(1) (1992),  21-42.

\bibitem {g0}Gould N.,  Orban D. and Toint Ph. \emph{CUTEst:
a constrained and unconstrained testing environment with safe threads
for mathematical optimization }, Comput. Optim. Appl. \textbf{60}(3) (2015),  545-557.

\bibitem{g1}Guanghui L.,  Jiye H. and  Hongxia Y. \emph{Global convergence of the Fletcher-Reeves algorithm with inexact linesearch}, Appl. Math. J. Chinese Univ. Ser,  \textbf{10}(1) (1995),  75-82.

\bibitem{h0}  Hager W. and  Zhang H. \emph{A new conjugate gradient method with muaranteed mescent and an efficient line search}, SIAM J. Optim.,  \textbf{16}, (2005),  170-192.

\bibitem{h1}  Hager W. and Zhang H. \emph{A survey of nonlinear conjugate gradient methods}, Pac. J. Optim., \textbf{2}(1) (2006),  35-58.

\bibitem{h2} Hestenes  M. and Stiefel E.  \emph{Methods of conjugate gradients for solving linear systems},
Journal of Research of the NIST,  \textbf{49} (1952),  409-436.

\bibitem{l01} Liu H. , Haijun W., Qian X. and Rao F. \emph{A conjugate gradient method with sufficient descent
property}, Numer. Algorithms,  \textbf{70}(2) (2015),  269-286.

\bibitem{l1} Liu  Y.  and  Storey C. \emph{Efficient generalized conjugate gradient algorithms Part 1: Theory}, J. Optim. Theory Appl.,  \textbf{69} (1991),  129-137.


\bibitem{l2} Livieris I., Tampakas V. and Pintelas P. \emph{A descent hybrid conjugate gradient method based on the memoryless BFGS update},  Numer. Algorithms,  \textbf{79}(4) (2018),  1169-1185.
\bibitem{n1} Nocedal J. and Wright S. \emph{ Numerical optimization}, Springer Science and Business Media, 2006.
\bibitem{p1} Polak E. and Ribiere G. \emph{Note sur la convergence de methodes de directions conjugues}, Rev. Française Informat. Recherche Operationnelle,  \textbf{3} (1969),  35-43.
\bibitem{s1}	 Shen J. and Mousavi S. \emph{Least sparsity of $p$-norm based
optimization problems with $p> 1$}, SIAM J. Optim. \textbf{28}(3) (2018),  2721-2751.

\bibitem{s2} Shewchuk J.R. \emph{An introduction to the conjugate gradient method without the agonizing pain}, Carnegie-Mellon University. Department of Computer Science, (1994).

\bibitem{y1} Yao S., He D.  and  Shi L. \emph{An improved Perry conjugate gradient method with adaptive parameter choice}, Numer. Algorithms,  (2018),  1-15.

\bibitem{z1} Zhang L., Zhou W. and  Li D.\emph{Global convergence of a modified Fletcher-Reeves conjugate gradient method with Armijo-type line search}, Numer. Math.,  \textbf{104}(4) (2006),  561-572.

\bibitem{z2}   Zoutendijk G. \emph{Nonlinear programming, computational methods},
Integer and nonlinear programming (1970),  37-86.

	 
	
\end{thebibliography}
\enddocument